\documentclass[11pt,reqno]{amsart}
\usepackage[margin=1.35in,letterpaper]{geometry}           
\usepackage{graphicx,soul}
\usepackage{amssymb,amsmath,amsthm,xcolor,mathrsfs,stmaryrd,accents,enumitem,hyperref}
\usepackage{epstopdf}

\usepackage{verbatim}
\makeatletter\let\over\@@over\makeatother

\newcommand{\be}{\begin{equation} }
\newcommand{\ee}{\end{equation}}
\newcommand{\bse}{\begin{subequations}}
\newcommand{\ese}{\end{subequations}}

\numberwithin{equation}{section}
\theoremstyle{plain} 
\newtheorem{theorem}{Theorem}[section] 
\newtheorem{proposition}{Proposition}[section]

\newtheorem{lem}{Lemma}[section]

\theoremstyle{remark}
\newtheorem{remark}{Remark}[section]

\newcommand{\LB}{\left[}
\newcommand{\RB}{\right]}
\newcommand{\LC}{\left(}
\newcommand{\RC}{\right)}
\newcommand{\LA}{\left<}
\newcommand{\RA}{\right>}

\newcommand{\R}{\mathbb{R}}

\usepackage{enumitem} 
\usepackage{hyperref}

\newcommand{\abs}[2][]{#1\lvert #2 #1\rvert}

\allowdisplaybreaks
\title[Transverse instability of CH-KP-I]{Transverse instability of the CH-KP-I equation}
\author[R. M. Chen]{Robin Ming Chen} 
\address{Department of Mathematics, University of Pittsburgh, Pittsburgh PA 15260}
\email{mingchen@pitt.edu}

\author[J. Jin]{Jie Jin}
\address{Department of Mathematics, University of Pittsburgh, Pittsburgh PA 15260}
\email{jij50@pitt.edu} 
\date{}

\begin{document}
\maketitle
\begin{abstract}
     The Camassa--Holm--Kadomtsev--Petviashvili-I equation (CH-KP-I) is a two dimensional generalization of the Camassa--Holm equation (CH). In this paper, we prove transverse instability of the line solitary waves under periodic transverse perturbations. The proof is based on the framework of  \cite{Rousset-Tzvetkov Poincare 2009}. Due to the high nonlinearity, our proof requires necessary modification. Specifically, we first establish the linear instability of the line solitary waves. Then through an approximation procedure, we prove that the linear effect actually dominates the nonlinear behavior.
\end{abstract}
\section{Introduction}
Surface water wave is too much of a monster to tame. Thus various asymptotic models have been developed to simplify it. In the realm of shallow water waves, these models include the KdV equation \cite{KortewegdeVries1895}, the Camassa--Holm equation \cite{Camassa-Holm 1993, Constantin Lannes 2009}, etc.. They are all unidirectional approximation models, which means that we assume the surface elevation is uniform in the transverse direction. A key observation is that these models all admit Hamiltonian structure, which indicates that it is reasonable to expect a systematic way to deal with a class of problems based on that structure. One problem focuses on the orbital stability around solitary waves  -- traveling waves which decay to zero at infinity. Roughly speaking, we want to know if the solution consistently stays in the neighborhood of a solitary wave and its translation  when its initial data does. A naive thinking why it is true is that the solitary wave holds the least Lagrangian action energy, so the object around it is ``willing" to evolve like that. One of the universal treatments is by center manifold theory. The center manifold theory is an equivalent but more algebraic form of the original problem (e.g. under Fourier transform), based on spectral decomposition. The ``finite dimension'' version of the spectral decomposition is purely algebraic in taste, while its corresponding ``infinite" counterpart has topology coming into play as a role of approximation to mimic the world of ``finite". This thought works well for some class of operators (e.g. normal operators), but not some others. For equations preserving the Hamiltonian structure, the linearized operator around a solitary wave has essential spectrum on the imaginary axis, which corresponds to center manifold part that is hard to deal with. Another treatment is by the Lyaponov method, which is by Benjamin \cite{Benjamin 1972} and Bona \cite{Bona1975}, and later generalized to handle a class of Hamiltonian models by Weinstein \cite{Weinstein1986} and Grillakis--Shatah--Strauss (GSS) \cite{GSS 1987}. They claim that knowing the information from the Lagrangian action energy allows one to determine the orbital stability and instability. The gain of their method is that instead of working with the original linearized operator, one just needs to study the spectrum of a rather transparent self-adjoint operator. The trade-off is that it is required to carefully weave the domain of the energy functional to balance between the complexity and solvability (due to loss of information from the original problem).

Besides the unidirectional models like KdV and CH, one can also allow transverse effect into modeling, leading to two-dimensional generalizations of the scalar models. Since the transverse perturbation is weak, it is natural to ask whether these models retain transverse stability, i.e. the unidirectional solitary waves remain stable under the two-dimensional flow. However, the answer to this question is much more involved. The first result is by Alexander--Pego--Sachs \cite{Alexander-Pego-Sachs 1997} on the Kadomtsev--Petviashvili (KP) equation
\begin{align*}
    (u_t+uu_x+u_{xxx})_x-\sigma u_{yy}=0
\end{align*}
which is a two-dimensional version of the KdV equation. The coefficient $\sigma$ takes values in $\{-1, 1\}$ representing the strength of capillarity relative to the gravitational forces. The weak surface tension case corresponds to $\sigma = 1$ and is referred to as the KP-I equation; and the strong surface tension leads to the so-called KP-II equation with $\sigma = -1$. In \cite{Alexander-Pego-Sachs 1997}, the authors state that the KP-I model is linearly stable, while the KP-II model is linearly unstable. The transition from linear instability to nonlinear instability for the KP-I equation is achieved by Rousset-Tzvetkov \cite{Rousset-Tzvetkov Poincare 2009}. Later on, they employed the same idea to a large class of equations \cite{Rousset-Tzvetkov JMPA 2008}. Transverse stability of the KP-II equation is proved by Mizumachi-Tzvetkov \cite{MK2012} and Mizumachi \cite{Mizumachi 2015}.

In this paper, we will study the Camassa--Holm--Kadomtsev--Petviashvili-I equation (CH-KP-I), which is a two-dimensional generalization of the Camassa-Holm equation (CH):
\begin{align}\label{CH-KP-I}
\LB\LC 1-\partial_x^2\RC u_t+3uu_x+2\kappa u_x-2u_xu_{xx}-uu_{xxx}\RB_x-u_{yy}=0
\end{align}
with $\kappa>0$.
In \cite{Chen 2006}, Chen derived a generalized version of \eqref{CH-KP-I} in the context of nonlinear elasticity theory. Also in \cite{Johnson 2002}, the CH-KP-II model is derived in the context of water wave. Note that in \eqref{CH-KP-I}, if we disregard the transverse effect, the CH-KP-I equation is reduced to the CH equation. The CH equation exhibits the wave-breaking phenomenon that is not shown in the KdV equation. From the point of view of modeling, this is because that these two models arise from different physical parameter regimes. More specifically, let $h$ and $\lambda$ denote respectively the mean elevation of the water over the bottom and the typical wavelength, and let $a$ be a typical wave amplitude. The parameter regime considered in the CH equation corresponds to 
\begin{equation*}
\varepsilon = {a\over h} \ll 1, \quad \delta = {h \over \lambda} \ll 1, \quad \varepsilon = O(\delta),
\end{equation*}
while the parameter regime for the KdV equation is $\varepsilon = O(\delta^2)$. Physically, $\varepsilon$ measures the strength of nonlinearity and $\delta$ characterizes the effect of dispersion, thus the CH equation possesses stronger nonlinearity than the KdV equation, which allows for the breaking wave. Like the KdV equation, solitary waves also exist for the CH equation, which are symmetric, monotone decreasing on positive $x$-axis and decay exponentially as $\abs{x}\rightarrow \infty$. Furthermore, the CH solitary waves are also orbitally stable like the KdV solitons, as is proved by Constantin--Strauss \cite{Constantin Strauss 2002} using the GSS method. For the CH-KP-I equation, since it could be treated as the CH counterpart of the two-dimensional
KdV equation (KP-I), it is reasonable to expect that the CH line solitary waves are also transversely unstable. Here a line solitary wave $\phi$ is defined such that it is uniform in the transverse direction, and for each cross section, it is exactly the solitary wave of the CH equation. The theorem we prove is as follows:
\begin{theorem}[Transverse instability of line solitary waves]\label{main theorem}
The CH line solitary wave $\phi$ of the CH-KP-I equation \eqref{CH-KP-I} is transversely unstable in the following sense: There exists $k_0>0$ such that  for every $s\geq 0$, there exists an  $\eta>0$ such that  for each $\delta>0$, there exists a solution $u^{\delta}$ emanating from an initial datum $u_0^{\delta}\in H^{\infty}(\mathbb{R}\times\mathbb{T}_a)$ with $\left\|u_0^{\delta}-\phi\right\|_{H^s(\R \times \mathbb{T}_a)}\leq\delta$, and a time $T^{\delta}\sim \left|\log\delta\right|$, so that $u^\delta$ satisfies
\begin{align}\label{instability quantity}
\inf_{l\in\mathbb{R}}\left\|u^{\delta}( T^{\delta},\cdot)-\phi(\cdot-l)\right\|_{L^2(\mathbb{R}\times\mathbb{T}_a)}\geq \eta,
\end{align}
where $a=\frac{2\pi}{k_0}$, $\mathbb{T}_a$ is the torus $\mathbb{R}/a\mathbb{Z}$.
\end{theorem}
Our proof is based on the pioneering work of Rousset--Tzvetkov \cite{Rousset-Tzvetkov Poincare 2009, Rousset-Tzvetkov JMPA 2008}. Their main idea is to first construct a most unstable eigenmode, and then prove that the nonlinear effect can be dominated by the linear effect, in the spirit of center manifold theory. The method works perfectly well for semilinear equations. However due to the nature of quasilinearity in our equation, we need to make necessary changes. The strategy is as follows: as in \cite{Rousset-Tzvetkov Poincare 2009, Rousset-Tzvetkov JMPA 2008}, the first step is to prove the linear instability by finding one unstable eigenvalue. Our method relies on \cite{Rousset-Tzvetkov 2010}. By taking Fourier transform with respect to $y$, the problem is transformed to finding a positive eigenvalue $\sigma$ corresponding to one frequency $k$. To handle this problem, it suffices to know the distribution of spectrum as $k$ evolves. The key issue is that for each $k$, the spectrum of the corresponding operator is hard to investigate compared with that of the KdV equation. Thus we have turned the problem to a generalized eigenvalue problem for a self-adjoint operator, and the spectrum of self-adjoint operator has much better property.

The second step is to prove the nonlinear instability based on the linear result. First, we choose the most unstable eigenmode $v^0$. Then we will prove that the solution $u^{\delta}=\phi+v^{\delta}$ with initial data $\phi+\delta v^0(0,\cdot)$ could lead to \eqref{instability quantity}. The estimate is based on the approximation procedure first constructed by Grenier \cite{Grenier 2000}. In details, the approximation of $v^{\delta}$ can be written as $v^{ap}=\delta\LC v^0+\sum_{k=1}^{M}\delta^{k}v^k\RC$. Since the nonlinearity of \eqref{CH-KP-I} is power-like, by matching the orders of $\delta$, it turns out that this approximate scheme is iterative. Unlike Picard iteration for the center manifold theory, each $v^k$ in this scheme solves a differential equation. The main reason why we choose this approximation scheme instead of the semigroup estimate is due to the high nonlinearity. For the semigroup estimates, since we couldn't have an explicit form of the semigroup, it is hard to conduct delicate analysis to close the energy estimates because of the loss of derivative. While for Grenier's approach, since for the $j$th iteration, $v^j$ is just a finite combination of the Fourier modes, it allows us to use energy estimates to overcome this difficulty. The rest of the proof consists of two parts. We first estimate $v^k$ and show that it can be controlled by $v^0$. Then an error estimate will follow. For the first part, by the Laplace transform, the original estimate for $v^k$ could be transformed to a resolvent estimate. The difficulty comes from higher order estimates. Compared with the KP-I equation in \cite{Rousset-Tzvetkov Poincare 2009}, \eqref{CH-KP-I} has stronger nonlinearity, and the corresponding linearized operator is weakly dispersive and nonlocal, making the energy estimates more challenging. What we do is to utilize the strong ``smoothing'' property together with a new cancellation mechanism resulting from the special structure of the nonlinearity. In this way, we are able to close the estimate at each iteration step. Finally the roughness of the energy estimates can be compensated by going to sufficiently high order approximation.

The rest of the paper is organized as follows. In Section \ref{Preliminary}, we present some notation, the Hamiltonian formulation and some preliminary results. In Section \ref{Linear Instability}, we will prove the linear instability. In Section \ref{Nonlinear Instability}, we will prove the nonlinear instability based on the linear instability. Several existence results will be given in the appendix.

\section{Preliminary}\label{Preliminary}
\subsection{Notation}\label{Notation}
We denote $|\cdot|_s$ for $\|\cdot\|_{H^s(\mathbb{R})}$ and $\|\cdot\|_s$ for $\|\cdot\|_{H^s(\mathbb{R}\times \mathbb{T}_a)}$, where $a=\frac{2\pi}{k_0}$ and $k_0$ will be given later. We also denote $\LA\cdot,\cdot\RA $ the inner product of $L^2(\mathbb{R})$. Finally, denote $\phi$ for $\phi_c$ for simplicity, where $\phi_c(x,y)$ is the line solitary wave of \eqref{CH-KP-I} with $\phi_c(x,y)=Q_c(x)$, and $Q_c$ represents the solitary wave of the CH equation with speed $c>2\kappa$. In the following, we will abuse using the notation of $\phi$ and $Q_c$ for convenience.
\subsection{Hamiltonian Formulation}\label{Mathematical Formulation}
The CH-KP-I equation \eqref{CH-KP-I} can be written in the Hamiltonian form:
\begin{align}\label{KP-CH Hamilton}
\begin{split}
u_t=\mathcal{J}\frac{\delta \mathcal{H}}{\delta u}=\LC 1-\partial_x^2\RC^{-1}\partial_x
\LC\frac{1}{2}u_x^2+uu_{xx}-2\kappa u-\frac{3}{2}u^2+\partial_x^{-2}\partial_y^2u\RC,
\end{split}
\end{align}
where 
\begin{align}\label{Hamiltonian}
\begin{split}
&\mathcal{J}=\LC 1-\partial_x^2\RC^{-1}\partial_x,\\ &\mathcal{H}=-\frac{1}{2}\int_{\mathbb{R}\times\mathbb{T}_a}\LB u^3+uu_x^2+2\kappa u^2-\LC\partial_x^{-1}\partial_yu\RC^2\RB dxdy,
\end{split}
\end{align}
and $\mathcal{H}$ is a conserved energy. A change of variable from $x$ to $x-ct$ yields that
\begin{align}\label{KP-CH Hamiltonian revised}
\begin{split}
u_t&=\mathcal{J}\frac{\delta(\mathcal{H}+c\mathcal{Q})}{\delta{u}}\\
&= \LC 1-\partial_x^2\RC^{-1}\partial_x\LC\frac{1}{2}u_x^2+uu_{xx}-2\kappa u-\frac{3}{2}u^2+\partial_x^{-2}\partial_y^2u+c(u-u_{xx})\RC,   
\end{split}
\end{align}
where $\mathcal{Q}=\frac{1}{2}\int_{\mathbb{R}\times\mathbb{T}_a}\LC u^2+u_x^2\RC dxdy$ is called the impulse which is another conserved quantity.
A line solitary wave $\phi$ with speed $c$ can be regarded as a critical point of $\mathcal{H}+c\mathcal{Q}$:
\begin{align}\label{solitary wave}
\frac{\delta (\mathcal{H}+c\mathcal{Q})}{\delta u}\LB\phi\RB=0.
\end{align}

\subsection{Preliminary Results}\label{Preliminary Results} We collect some results that will be used later. 
\begin{proposition}[\cite{Constantin Strauss 2002}]\label{solitary wave property}
The line solitary wave $\phi$ with speed $c>2\kappa$ satisfies the following properties:
\begin{enumerate}
    \item It is smooth, positive, even and decreasing from its peak of height $c-2\kappa$.
    \item It is concave when $\phi\in \LC c-\frac{\kappa}{2}-\sqrt{c\kappa+\frac{\kappa^2}{4}},c-2\kappa\RC$ and convex elsewhere.
    \item $\phi \sim \exp\LC-\sqrt{1-\frac{2\kappa}{c}}\abs{x}\RC $\ \ \text{for} \ $\abs{x}\rightarrow \infty$.
\end{enumerate}
\end{proposition}
\begin{theorem}[\cite{Constantin Strauss 2002}]\label{CH linearized operator}
For the linearized operator $H_c$ of the CH equation about the solitary wave $\phi$: $H^1(\mathbb{R})\rightarrow H^{-1}(\mathbb{R})$,
\begin{align}\label{H_c}
    H_c= -\partial_x\LC(c-\phi)\partial_x\RC+\phi''-3\phi-2\kappa+c,
\end{align}
it has exactly one simple negative eigenvalue, one simple zero eigenvalue and the rest of the spectrum is positive and bounded away from zero.
\end{theorem}
\section{Linear Instability}\label{Linear Instability}
In this section, we will first prove the linear instability, from which we will construct a most unstable eigenmode in the next section. Denote $v=u-\phi$, the linearized equation of \eqref{KP-CH Hamiltonian revised} about $\phi$ is 
\begin{align}\label{linearized equation}
\partial_tv=\mathcal{J}\mathcal{L}v,
\end{align}
where 
\begin{align}\label{linear Hamiltonian}
\mathcal{L}=-\partial_x\LC(c-\phi)\partial_x\RC+\LC\phi''-3\phi-2\kappa+c\RC+\partial_x^{-2}\partial_y^2.  
\end{align}
Let $v=e^{\sigma t}e^{ik y}U$, then
\begin{align*}
\sigma U=\mathcal{J}(k)\circ \mathcal{L}(k)U,  
\end{align*}
where 
\begin{align*}
\begin{split}
\mathcal{L}(k)=e^{-ik y}\mathcal{L}e^{ik y}
=-\partial_x((c-\phi)\partial_x)+\LC\phi''-3\phi-2\kappa+c\RC-k^2\partial_x^{-2},\quad \mathcal{J}(k)=\mathcal{J}.    
\end{split}
\end{align*}
Let $U=\mathcal{J}^*(k)W$. Then
\begin{align}\label{Generalized eigen}
\sigma\mathcal{J}(k)^*W=\tilde{\mathcal{L}}(k) W,  
\end{align}
where $\tilde{\mathcal{L}}(k)=\mathcal{J}(k)\mathcal{L}(k)\mathcal{J}^*(k)$. The proof of the linear instability is based on the following theorem:
\begin{theorem}[\cite{Rousset-Tzvetkov 2010}]\label{RT2010}
Assume the following conditions:
\begin{enumerate}
    \item There exist $K>0$ and $\alpha>0$ such that  $\tilde{\mathcal{L}}(k)\geq \alpha\text{Id}$ for $|k|\geq K$;
    \item The essential spectrum of $\tilde{\mathcal{L}}(k)$ is included in $[c_k,+\infty)$ with $c_k>0$ for $k\neq 0$;
    \item For every $k_1\geq k_2\geq 0$, we have $\tilde{\mathcal{L}}(k_1)\geq \tilde{\mathcal{L}}(k_2)$. In addition, if for some $k>0$ and $U\neq 0$, we have $\tilde{\mathcal{L}}(k)U=0$, then $\LA\tilde{\mathcal{L}}'(k)U,U\RA >0$;
    \item The spectrum of $\tilde{\mathcal{L}}(0)$ is under the form $\{-\lambda\}\cup I$ where $-\lambda<0$ is an isolated simple eigenvalue and $I$ is included in $[0,+\infty)$.
\end{enumerate}
Then there exist $\sigma >0$, $k\neq 0$ and $U$ solving \eqref{Generalized eigen}.
\end{theorem}
\begin{proposition}
(Existence of an unstable eigenmode)\label{unstable eigenmode}
If $c>2\kappa>0$, there exists one unstable eigenmode for  \eqref{linearized equation}.
\end{proposition}
\begin{proof}
According to Theorem \ref{RT2010}, it suffices to verify conditions (1)-(4) for $\tilde{\mathcal{L}}(k):\ H^2(\mathbb{R})\rightarrow L^2(\mathbb{R})$.\\
(1) It is easy to see that 
\begin{align*}
\begin{split}
\LA\tilde{\mathcal{L}}(k)u,u\RA = & \LA(c-\phi)\LC 1-\partial_x^2\RC^{-1}u_{xx},\LC 1-\partial_x^2\RC^{-1}u_{xx}\RA \\
& +\LA\LC\phi''-3\phi-2\kappa+c\RC\LC 1-\partial_x^2\RC^{-1}u_x,\LC 1-\partial_x^2\RC^{-1}u_x\RA \\
& +k^2\LA\LC 1-\partial_x^2\RC^{-1}u,\LC 1-\partial_x^2\RC^{-1}u\RA \\
\geq & \ \kappa \left|\LC 1-\partial_x^2\RC^{-1}u_{xx}\right|_0^2-\alpha \left|\LC 1-\partial_x^2\RC^{-1}u_x\right|_{0}^2+k^2 \left|\LC 1-\partial_x^2\RC^{-1}u\right|_{0}^2.    
\end{split}
\end{align*}
Here we require $\kappa>0$ and have used Proposition \ref{solitary wave property}. Then since $\left|\LC 1-\partial_x^2\RC^{-1}u_x\right|_{0}$ can be controlled by $\left|\LC 1-\partial_x^2\RC^{-1}u_{xx}\right|_{0}$ and $\left|\LC 1-\partial_x^2\RC^{-1}u\right|_{0}$, (1) is verified.\\
(2)
Consider 
$n(k):\ \ H^4(\mathbb{R})\rightarrow L^2(\mathbb{R})$  as
\begin{align*}
\tilde{\mathcal{L}}(k)=\LC 1-\partial_x^2\RC^{-1}n(k)\LC 1-\partial_x^2\RC^{-1},   
\end{align*}
thus 
\begin{align*}
n(k)=-\partial_x\LC-\partial_x\LC(c-\phi)\partial_x\RC+\LC\phi''-3\phi-2\kappa+c\RC\RC\partial_x+k^2.   
\end{align*}
It can be seen that $\tilde{\mathcal{L}}(k)$ and $n(k)$ have the same Fredholm index since $\LC 1-\partial_x^2\RC^{-1}: H^{s-2}(\mathbb{R}) \rightarrow H^s(\mathbb{R})$ has Fredholm index $0$. Thus the essential spectrum of $\tilde{\mathcal{L}}(k)$ and $n(k)$ are the same. By Weyl's lemma, we just need to study the essential spectrum of the limiting operator 
\begin{align*}
n_{\infty}(k)=c\partial_x^4-(c-2\kappa)\partial_x^2+k^2. 
\end{align*}
Consider $(n_{\infty}(k)-\beta)u=f$, since $c>2\kappa$. Using Fourier transform, we have the essential spectrum lying in $[c_k,\infty)$ for some $c_k>0$. \\
(3) Direct computation yields
\begin{align*}
\LA\tilde{\mathcal{L}}'(k)u,u\RA =2k\LA \LC 1-\partial_x^2\RC^{-1}u,\LC 1-\partial_x^2\RC^{-1}u\RA >0\ \ \mbox{for}\ k>0.    
\end{align*}
(4)
The proof follows from the discussion on KP-I in \cite{Rousset-Tzvetkov 2010}. Observe from \eqref{H_c} that we can write $\tilde{\mathcal{L}}(0)$ as
$-\mathcal{J}H_c\mathcal{J}$. By Theorem \ref{CH linearized operator}, it has a unique simple negative eigenvalue with the associated eigenvector $\psi$. By the approximation argument, we have $\mathcal{J}u_n$ tending to $\psi$, and
\begin{align*}
\LA \tilde{\mathcal{L}}(0)u_n,u_n\RA =\LA  H_c\mathcal{J}u_n,\mathcal{J}u_n\RA <0    
\end{align*}
for $n$ sufficiently large. Thus $\tilde{\mathcal{L}}(0)$ has at least one negative eigenvalue. On the other hand, for $u$ with $\LA \mathcal{J} u,\psi\RA =0$, we have $\LA \tilde{\mathcal{L}}(0)u,u\RA =\LA  H_c\mathcal{J} u,\mathcal{J} u\RA \geq 0$. Thus we conclude that $\tilde{\mathcal{L}}(0)$ just has one negative eigenvalue which is simple.
\end{proof}
\section{Nonlinear Instability}\label{Nonlinear Instability}
\subsection{Construction of a most unstable eigenmode}\label{Construction of a most unstable eigenmode}
As discussed in the previous section, there exists an unstable mode $k_0\neq 0$ which indicates the linear instability. In the rest of the paper, we consider the period with respect to $y$ to be $a=\frac{2\pi}{k_0}$. Let 
\begin{align*}
v=e^{\sigma t}e^{im k_0y}U(m,x), \ \ \ m\in \mathbb{Z},
\end{align*}
which solves $\partial_tv=\mathcal{J}\mathcal{L}v$. By Fourier transforming with respect to $y$, we have
\begin{align}\label{find most eigenmode}
\sigma U=\mathcal{J}\mathcal{L}(mk_0)U.
\end{align}

The construction of a most unstable eigenmode is based on the following lemma:
\begin{lem}[\cite{Rousset-Tzvetkov JMPA 2008}]\label{RT JMPA 2008}
Consider the problem \eqref{find most eigenmode}. There exists $K>0$ such that  for $|mk_0|\geq K$, there is no nontrivial solution with $Re(\sigma)\neq 0$. In addition, for every $k\neq 0$, there is at most one unstable mode with corresponding transverse frequency $k$.
\end{lem}
\begin{remark}
The proof of Lemma \ref{RT JMPA 2008} is based on the fact that $\mathcal{L}(mk_0)$ is positive definite, which is easy to check.
\end{remark}
According to Lemma \ref{RT JMPA 2008}, $\sigma_0, U_0$ can be chosen corresponding to the maximal $m_0$, and the most unstable eigenmode $v^0$ can be written as 
\begin{align*}
v^0=2\text{Re}( e^{\sigma_0t}e^{im_0k_0y}U_0).    
\end{align*}

We now construct the unstable solution $u^{\delta}$ with initial data $\phi+\delta v^0(0,x,y)$. Let $v=u^{\delta}-\phi$, then it satisfies 
\begin{align}\label{nonlinear equation}
\partial_tv=\mathcal{J}\mathcal{L}v+\mathcal{J}\LC\frac{1}{2}v_x^2+vv_{xx}-\frac{3}{2}v^2\RC,\quad v(0,x,y)=\delta v^0(0,x,y).
\end{align}
Thus in order to prove the nonlinear instability of \eqref{CH-KP-I}, it suffices to study the behavior of $v$.
\subsection{Construction of a high order unstable approximate solution} \label{Construction of a high order unstable approximate solution}
Define $V_K^s$ as the following truncated space:
\begin{align*}
V_K^s=\left\{u:\ u=\sum_{j=-K}^{j=K}u_je^{ijm_0k_0y},\ u_j\in H^s(\mathbb{R})\right\} , 
\end{align*}
where the norm on $V_K^s$ is defined by $|u|_{V_K^s}=\mbox{sup}_j|u_j|_{s}$. It can be seen that $v^0\in V_1^s$ for all $s\in\mathbb{N}$. We look for a high order approximate solution 
\begin{align*}
v^{ap}=\delta\LC v^0+\sum_{k=1}^{M}\delta^{k}v^k\RC,\quad v^k\in V_{k+1}^{s-k}.
\end{align*}
By matching the orders of $\delta$, it yields that
\begin{equation}\label{approximate equation}
\left\{\begin{aligned}
\partial_tv^k=&\mathcal{J}\mathcal{L}v^k+\mathcal{J}\LB\frac{1}{2}\LC\sum_{j+l=k-1}v_x^jv_x^l\RC+\sum_{j+l=k-1}v^jv_{xx}^l-\frac{3}{2}\sum_{j+l=k-1}v^jv^l\RB,\\ v^k|_{t=0}&=0    
\end{aligned}
\right.
\end{equation}
for $1\leq k\leq M$.
\begin{proposition}\label{v estimate}
Let $v^k$ be the solution of \eqref{approximate equation}, if $s-k\geq 0$ then
\begin{align}
\big|v^k\big|_{V_{k+1}^{s-k}} \leq C_{M,s}e^{(k+1)\sigma_0t},  
\end{align}
where $C_{M,s}>0$ depends on the approximation order $M$ and regularity $s$.
\end{proposition}
\begin{remark}
This proposition implies that the effect of $v^k$ can be controlled by $v^0$.
\end{remark}
Indeed, the above proposition can be easily derived by induction from the following theorem:
\begin{theorem}\label{u estmate}
Consider the solution $u$ of the linear problem 
\begin{align}\label{linear problem}
\partial_tu=\mathcal{J}\mathcal{L}u+\mathcal{J}F    
\end{align}
with $F\in V_{K}^{s-1}$ and 
\begin{align*}
|F|_{V_{K}^{s-1}} \leq C_{K,s}e^{\gamma t},\ \gamma\geq 2\sigma_0,
\end{align*}
then $u\in V_{K}^s$ and satisfies the estimate 
\begin{align*}
|u|_{V_K^s}\leq C_{K,s}e^{\gamma t}.  
\end{align*}
\end{theorem}
By Fourier transforming with respect to $y$, we have
\begin{align}\label{linear problem Fourier}
\partial_tu_j=\mathcal{J}\mathcal{L}(jm_0k_0)u_j+\mathcal{J}F_j,\quad u_j|_{t=0}=0.\ \ \ \text{for}\ j=1\cdots, K,
\end{align}
where $u_j, F_j$ are the $j$th Fourier modes  in $y$ of $u$ and $F$ respectively. Thus the problem is equivalent to proving that if 
\begin{align}\label{fj estimate}
|F_j|_{V_{K}^{s-1}} \leq C_{K,s}e^{\gamma t},\ \gamma\geq 2\sigma_0,\ \ \ \text{for}\ j=1\cdots, K,
\end{align}
then
\begin{align}\label{uj estimate}
|u_j|_{V_K^s}\leq C_{K,s}e^{\gamma t},\ \ \ \text{for}\ j=1\cdots, K.
\end{align}
\begin{lem}[Existence of $u_j$]\label{well posedness u_j} For any $s\in \mathbb{R}$ , there exists a unique $u_j \in H^s(\mathbb{R})$ solving \eqref{linear problem Fourier}.
\end{lem}
\begin{proof}
The proof of Lemma \ref{well posedness u_j} is postponed in Appendix \ref{appendix lem}.
\end{proof}

To prove \eqref{uj estimate}, we first give a resolvent estimate. Take $\gamma_0$ such that  $\sigma_0<\gamma_0<\gamma$. For $T>0$, we introduce 
\begin{align}
G=0,\ t<0;\qquad G=0,\ t>T; \qquad G=F_j,\ t\in [0,T],    
\end{align}
then \eqref{linear problem Fourier} can be written as 
\begin{align*}
\partial_t\tilde{u}_j=\mathcal{J}\mathcal{L}(jm_0k_0)\tilde{u}_j+\mathcal{J}G,\quad \tilde{u}_j|_{t=0}=0.    
\end{align*}
where $\tilde{u}_j $ is the extension of $u_j$ such that 
\begin{align*}
\tilde{u}_j|_{0\leq t\leq T}=u_j,\quad \tilde{u}_j|_{t<0}=0.
\end{align*}
Then the Laplace transform yields that 
\begin{align}\label{resolvent equation}
(\gamma_0+i\tau) w=\mathcal{J}\mathcal{L}(jm_0k_0)w+\mathcal{J}H,   
\end{align}
where
\begin{align*}
w=\int_{t\geq 0}e^{-(\gamma_0+i\tau)}\tilde{u}_jdt,\qquad
H=\int_{t\geq 0}e^{-(\gamma_0+i\tau)}Gdt.    
\end{align*}
Here  for simplicity, we denote $w$ as the  Laplace transform of $\tilde{u}_j$ for each given $j$.
\begin{theorem}[Resolvent estimate]\label{resolvent estimate}
Let $s\geq 1$. Let $w$ be the solution of \eqref{resolvent equation}, then there exists a constant $C(s,\gamma_0,K)$ such that  for every $\tau$, we have the estimate
\begin{align}
|w|_{s}^2\leq C(s,\gamma_0,K)|H|^2_{s-1}.
\end{align}
\end{theorem}
We will split the proof of the above theorem into Lemma \ref{tau large} and Lemma \ref{tau small}. 
\begin{lem}\label{tau large}
There exist $M>0$ and $C(s,\gamma_0,K)$ such that  for $|\tau|\geq M$, we have 
\begin{align}\label{tau large estimate}
|w|_{s}^2\leq C(s,\gamma_0,K)|H|^2_{s-1}.
\end{align}
\end{lem}
\begin{proof}
First prove the case when $s=1$. Write
\begin{align*}
\mathcal{L}(jm_0k_0)=L-(jm_0k_0)^2\partial_x^{-2}    
\end{align*} 
where
\begin{align*}
L=-\partial_x((c-\phi)\partial_x)+\LC\phi''-3\phi-2\kappa+c\RC.
\end{align*}
Then we decompose
\begin{align}\label{w decompose} 
w=\alpha\phi_{-1}+\beta\phi_0+w_{\perp}    
\end{align}
such that  
\begin{align}\label{CH decomposition}
L\phi_{-1}=\mu\phi_{-1},\ \mu<0;\quad L\phi_0=0;\quad \LA  Lw_{\perp},w_{\perp}\RA \geq c_{\perp}w_{\perp}^2,\ c_{\perp}>0.
\end{align}
By Theorem \ref{CH linearized operator}, such a decomposition is available. Taking the inner product of \eqref{resolvent equation} with
$\mathcal{L}(jm_0k_0)$ yields that 
\begin{equation}\label{inner product resolvent}
\gamma_0\LC\LA  w,Lw\RA +(jm_0k_0)^2\left|\partial^{-1}_xw\right|_0^2\RC=\text{Re}\LC\LA \mathcal{J}H,Lw\RA + \LA \mathcal{J}H,(jm_0k_0)^2\partial_x^{-2}w\RA \RC.
\end{equation}
By \eqref{CH decomposition} and \eqref{inner product resolvent}, we have
\begin{align*}
\gamma_0\LC\mu\alpha|\phi_{-1}|_0^2+c_{\perp}|w_{\perp}|_0^2+\LC jm_0k_0\RC^2\left|\partial_x^{-1}w\right|_0^2\RC\leq C\LC|H|_0|w|_{1}+(jm_0k_0)^2|H|_{-2}\left|\partial_x^{-1}w\right|_0\RC,
\end{align*}
then
\begin{align}\label{alpha beta 1}
|w_{\perp}|_0^2+(jm_0k_0)^2\left|\partial_x^{-1}w\right|_0^2\leq C\LC|\alpha|^2+|H|_{-2}^2+|H|_0|w|_1\RC.
\end{align}
Taking the inner product of \eqref{resolvent equation} with $\phi_{-1}$ and $\phi_0$ respectively, we have 
\begin{align*}
\begin{split}
&(\gamma_0+i\tau)\alpha=-\LA  w,L\mathcal{J}\phi_{-1}\RA -(jm_0k_0)^2\LA \mathcal{J}\partial_x^{-2}w,\phi_{-1}\RA +\LA \mathcal{J}H,\phi_{-1}\RA, \\
&(\gamma_0+i\tau)\beta=-\LA  w,L\mathcal{J}\phi_{0}\RA -(jm_0k_0)^2\LA \mathcal{J}\partial_x^{-2}w,\phi_{0}\RA +\LA \mathcal{J}H,\phi_{0}\RA.   
\end{split}
\end{align*}
Rewriting $w$ as \eqref{w decompose} for the first term on the right-hand side and combining the above two equations, we have 
\begin{align}\label{alpha beta 2}
\LC\gamma_0+|\tau|\RC\LC|\alpha|+|\beta|\RC\leq C\LC|\alpha|+|\beta|+|w_{\perp}|_0+(jm_0k_0)^2\left|\partial_x^{-1}w\right|_0+|H|_{-2}\RC.  
\end{align}
Multiplying $|\alpha|+|\beta|$ to \eqref{alpha beta 2} and by Cauchy-Schwarz inequality it follows that
\begin{align*}
(\gamma_0+|\tau|-C)\LC|\alpha|^2+|\beta|^2\RC\leq C \LC|w_{\perp}|_0^2+(jm_0k_0)^2\left|\partial_x^{-1}w\right|_0^2+|H|_{-2}^2\RC,    
\end{align*}
which is a good estimate when $|\tau|$ is large. For a sufficiently large constant $B$, consider $B$\eqref{alpha beta 1}+\eqref{alpha beta 2}:
\begin{align*}
\begin{split}
&(B-C)\LC|w_{\perp}|_0^2+(jm_0k_0)^2\left|\partial_x^{-1}w\right|_0^2\RC+
\LC\gamma_0+|\tau|-BC-C\RC\LC|\alpha|^2+|\beta|^2\RC\\
&\quad\leq\ BC\LC |H|_{-2}^2+|H|_0|w|_{1}\RC.
\end{split}
\end{align*}
When $|\tau|> C+BC$ we have 
\begin{align}\label{alpha beta 3}
|w|_0^2+(jm_0k_0)^2\left|\partial_x^{-1}w\right|_0^2\leq C\LC|H|_0|w|_{1}+|H|_{-2}^2\RC.
\end{align}
On the other hand,
\begin{align}\label{lyapunov type}
\begin{split}
\LA  w,Lw\RA &=\LA  w, -\partial_x\LC(c-\phi)\partial_xw\RC+\LC\phi''-3\phi-2\kappa+c\RC w\RA \\
& \geq a_1|w|_1^2+\LA \LC\phi''-3\phi-2\kappa+c\RC w,w\RA     
\end{split}
\end{align}
for $a_1>0$. Replacing $\LA  w,Lw\RA $ in \eqref{inner product resolvent} with \eqref{lyapunov type}, we have 
\begin{align}\label{alpha beta 4}
|w|_{1}^2+(jm_0k_0)^2\left|\partial_x^{-1}w\right|_0^2\leq C\LC|w|_0^2+|H|_{-2}^2+|H|_0|w|_{1}\RC.
\end{align}
Combining \eqref{alpha beta 3} and \eqref{alpha beta 4} yields
\begin{align*}
|w|_{1}^2+(jm_0k_0)^2\left|\partial_x^{-1}w\right|_0^2\leq C\LC|H|_{-2}^2+|H|_0|w|_{1}\RC.
\end{align*}
Consequently, 
\begin{align*}
  |w|_{1}^2+(jm_0k_0)^2\left|\partial_x^{-1}w\right|_0^2\leq C|H|_0^2,
\end{align*}
which proves the case $s=1$.

For higher order estimates, \eqref{resolvent equation} can be written as
\begin{align}\label{high order equation}
\begin{split}
(\gamma_0+i\tau)w=&\LC 1-\partial_x^2\RC ^{-1}\LC -\partial_x^2\RC \LB(c-\phi)w_x\RB\\
&+\LC 1-\partial_x^2\RC ^{-1}\partial_x\LB\LC\phi''-3\phi-2\kappa+c\RC w\RB\\
&-(jm_0k_0)^2\LC 1-\partial_x^2\RC ^{-1}\partial_x^{-1}w+\LC 1-\partial_x^2\RC ^{-1}\partial_xH.   
\end{split}
\end{align}
For the first term on the right-hand side, we can rewrite it as
\begin{align}\label{important transformation}
\LC 1-\partial_x^2\RC ^{-1}\LC -\partial_x^2\RC \LC (c-\phi)w_x\RC =(c-\phi)w_x-\LC 1-\partial_x^2\RC ^{-1}\LC (c-\phi)w_x\RC.     
\end{align}
By induction, assume \eqref{tau large estimate} is true for $s$. We prove that it is true for $s+1$. In the rest of this proof, we denote $O(\partial_x^k w)$ as generic polynomial differential operator on $w$ with highest degree $k$.

Take the inner product of \eqref{high order equation} with 
\begin{align*}
(-1)^{s+1}\partial_x^{2s+2}w+(-1)^{s+1}\partial_x^{s+1}\LC r_{s+1}(x)\partial_x^{s+1}w\RC,
\end{align*}
where $r_{s+1}(x)$ is bounded which will be determined later.
We have 
\begin{align}
&\text{Re}\LA (\gamma_0+i\tau) w,(-1)^{s+1}\partial_x^{2s+2}w\RA  \label{rs 0}\\
&\quad=\ \gamma_0|w|_{s+1}^2, \notag\\
&\text{Re}\LA (c-\phi)w_x,(-1)^{s+1}\partial_x^{2s+2}w\RA  \label{rs 1}\\
&\quad=\ \text{Re}\LA \LC\partial^{s+1}(c-\phi)w_x\RC,\partial_x^{s+1}w\RA  \notag\\
&\quad=\ \text{Re}\LA  (c-\phi)\partial_x^{s+2}w-(s+1)\phi'\partial_x^{s+1}w+O(\partial_x^sw),\partial_x^{s+1}w\RA \notag\\
&\quad=\ \LA \LC\frac{1}{2}-(s+1)\RC\phi'\partial_x^{s+1}w+O(\partial_x^sw),\partial_x^{s+1}w\RA \notag\\
&\quad=\ -\LA \LC s+\frac{1}{2}\RC\phi'\partial_x^{s+1}w+O(\partial_x^sw),\partial_x^{s+1}w\RA ,\notag\\
\text{and}\ \ \ \notag\\
&\text{Re}\LA (\gamma_0+i\tau)w,(-1)^{s+1}\partial_x^{s+1}\LC r_{s+1}(x)\partial_x^{s+1}w\RC\RA  \label{rs 2}\\
&\quad=\ \gamma_0\LA  r_{s+1}(x)\partial_x^{s+1}w,\partial_x^{s+1}w\RA ,      
\notag\\
&\text{Re}\LA  (c-\phi)w_x,(-1)^{s+1}\partial_x^{s+1}\LC r_{s+1}(x)\partial_x^{s+1}w\RC\RA \label{rs 3}\\
&\quad=\ \text{Re}\LA  r_{s+1}(x)\partial_x^{s+1}\LC(c-\phi)w_x\RC,\partial_x^{s+1}w\RA \notag\\
&\quad=\ \text{Re}\Big<  r_{s+1}(x)\big( ( c-\phi)\partial_x^{s+2}w-(s+1)\phi'r_{s+1}(x)\partial_x^{s+1}w\notag\\
&\quad\quad+O(\partial_x^sw)\big),\partial_x^{s+1}w\Big> \notag\\
&\quad=\ \text{Re}\LA  \LC\frac{1}{2}\LC\phi'r_{s+1}(x)-r_{s+1}'(x)(c-\phi)\RC-(s+1)\phi'r_{s+1}(x)\RC\partial_x^{s+1}w \right. \notag\\
&\quad\quad +O(\partial_x^sw),\partial_x^{s+1}w\Big> ,\notag\\
\text{and}\ \ \ \notag\\
&\LA \LC 1-\partial_x^2\RC^{-1}((c-\phi)w_x),(-1)^{s+1}\partial_x^{2s+2}w+(-1)^{s+1}\partial_x^{s+1}\LC r_{s+1}(x)\partial_x^{s+1}w\RC\RA  \label{rs 4}\\
&\quad=\ \LA  O(\partial_x^sw),\partial_x^{s+1}w\RA ,\notag\\   
&\LA \LC 1-\partial_x^2\RC^{-1}\partial_x\LC\LC\phi''-3\phi-2\kappa+c\RC w\RC,(-1)^{s+1}\partial_x^{2s+2}w+(-1)^{s+1}\partial_x^{s+1}\LC r_{s+1}(x)\partial_x^{s+1}w\RC\RA \notag\\
&\quad=\ \LA  O(\partial_x^sw),\partial_x^{s+1}w\RA ,\notag\\
&\LA  (jm_0k_0)^2\LC 1-\partial_x^2\RC^{-1}\partial_x^{-1}w,(-1)^{s+1}\partial_x^{2s+2}w+(-1)^{s+1}\partial_x^{s+1}\LC r_{s+1}(x)\partial_x^{s+1}w\RC\RA \notag\\
&\quad=\ \LA  O(\partial_x^{s-1}w),\partial_x^{s+1}w\RA ,\notag\\
&\LA  \LC 1-\partial_x^2\RC^{-1}\partial_xH,(-1)^{s+1}\partial_x^{2s+2}w+(-1)^{s+1}\partial_x^{s+1}\LC r_{s+1}(x)\partial_x^{s+1}w\RC\RA \notag\\
&\quad=\ \LA  O(\partial_x^sH),\partial_x^{s+1}w\RA. \notag
\end{align}
We want to use $r_{s+1}(x) $ to eliminate $-(s+\frac{1}{2})\phi'$ in \eqref{rs 1} with \eqref{rs 2}, \eqref{rs 3}. On the other hand, since $\phi'\rightarrow 0$ when $|x|\rightarrow \infty$, by \eqref{rs 1}, there exists $A>0$ such that  $\gamma_0+(s+\frac{1}{2})\phi'>0$ when $|x|>A$. Then we want $r_{s+1}(x)$ to satisfy when $x>-A$, 
\begin{align}\label{rs equation}
\begin{split}
-\LC s+\frac{1}{2} \RC \phi' -\gamma_0r_{s+1}(x) & +\frac{1}{2}\LC\phi'r_{s+1}(x)-r_{s+1}'(x)(c-\phi)\RC\\
&-(s+1)\phi'r_{s+1}(x)=0, \end{split}
\end{align}
and when $x<-A$
\begin{align}
\begin{split}
-\gamma_0r_{s+1}(x) & +\frac{1}{2}(\phi'r_{s+1}(x)-r_{s+1}'(x)(c-\phi))\\ &-(s+1)\phi'r_{s+1}(x) \leq  \gamma_0+(s+\frac{1}{2})\phi'.     
\end{split}
\end{align}
One choice could be $r_{s+1}(x)=0\ \text{when}\ x\leq -A$ and $r_{s+1}(x)$ satisfies \eqref{rs equation} when $x>-A$. Note that \eqref{rs equation} can be written in a form of Bernoulli equation:
\begin{align*}
r_{s+1}'(x)+\frac{2\gamma_0+(2s+1)\phi'}{c-\phi}r_{s+1}(x)=-(2s+1)\frac{\phi'}{c-\phi}.   
\end{align*}
So when $x>-A$  \\
\begin{align*}
r_{s+1}(x)=-(2s+1)e^{-\int_{-A}^x\frac{2\gamma_0+(2s+1)\phi'}{c-\phi}}
\int_{-A}^x  e^{\int_{-A}^t\frac{2\gamma_0+(2s+1)\phi'}{c-\phi}}\frac{\phi'}{c-\phi}dt
\end{align*}
and $r_{s+1}(s)$ is bounded. Indeed, when $x\rightarrow +\infty$, $\frac{2\gamma_0+(2s+1)\phi'}{c-\phi}$ is positive, and the forcing term $-(2s+1)\frac{\phi'}{c-\phi}\rightarrow 0$, which will prevent $\abs{r_{s+1}(x)}\rightarrow \infty$.

So by \eqref{rs 0}-\eqref{rs 4} and \eqref{rs equation} 
\begin{align*}
\gamma_0|w|_{s+1}^2=\text{Re}\LC\LA  O(\partial_x^sw),\partial_x^{s+1}w\RA +\LA  O(\partial_x^sH),\partial_x^{s+1}w\RA \RC.    
\end{align*}
Since  $|w|_{k}$ is bounded by $|w|_{s+1}$ and $|w|_1$ for $1<k\leq s$, by Cauchy-Schwartz inequality
\begin{align*}
|w|_{s+1}^2\leq C|H|_s^2,    
\end{align*}
which proves the lemma.
\end{proof}
\begin{lem}\label{tau small}
For $|\tau|\leq M$, we have 
\begin{align*}
|w|_{s}^2\leq C(s,\gamma_0,K,M)|H|^2_{s-1}.
\end{align*}
\end{lem}
Write $\sigma=\gamma_0+i\tau$ and impose $\LC 1-\partial_x^2\RC\partial_x$ on \eqref{resolvent equation}
\begin{align*}
\begin{split}
\sigma\LC 1-\partial_x^2\RC w_x=&-\partial_x^3\LC(c-\phi)w_x\RC+\partial_x^2\LC(\phi''-3\phi-2\kappa+c)w\RC\\
&-(jm_0k_0)^2w+\partial_x^2H.    
\end{split}
\end{align*}
Then
\begin{align*}
\begin{split}
(c-\phi)\partial_x^4w=\ &\LC 3\phi'+\sigma\RC\partial_x^3w+\LC3\phi''+\LC\phi''-3\phi-2\kappa+c\RC\RC\partial_x^2w\\
&+\LC\phi'''-\sigma+2\LC\phi''-3\phi-2\kappa+c\RC'\RC\partial_xw\\
&+\LC\LC\phi''-3\phi-2\kappa+c\RC''-(jm_0k_0)^2\RC w +\partial_x^2H .   
\end{split}
\end{align*}
Let $V=\LC w,\partial_xw,\partial_x^2w,\partial_x^3w\RC^{T}$. We have 
\begin{align*}
\frac{dV}{dx}=A(x,\sigma,j)V+\partial_x^2H.  
\end{align*}
Here 
\begin{align*}
A(x,\sigma,j)=\frac{1}{c-\phi} 
\LC\begin{matrix}
0 & c-\phi & 0 & 0 \\
0 & 0 & c-\phi & 0 \\
0 & 0 & 0 & c-\phi \\
A_{41} & A_{42}
& A_{43} & A_{44}
\end{matrix}\RC,
\end{align*}
where 
\begin{align*}
\begin{split}
    A_{41}&=\LC\phi''-3\phi-2\kappa+c\RC''-(jm_0k_0)^2,\\
    A_{42}&=\phi'''-\sigma+2\LC\phi''-3\phi-2\kappa+c\RC',\\
    A_{43}&=3\phi''+\LC\phi''-3\phi-2\kappa+c\RC,\\
    A_{44}&=3\phi'+\sigma,
\end{split}
\end{align*}
and the limiting matrix
\begin{align*}
A_{\infty}(\sigma,j)=\frac{1}{c}
\begin{pmatrix}
0 & c & 0 & 0 \\
0 & 0 & c & 0 \\
0 & 0 & 0 & c \\
-(jm_0k_0)^2 & -\sigma & -2\kappa+c & \sigma
\end{pmatrix}.
\end{align*}

The proof of Lemma \ref{tau small} is based on the following lemma:
\begin{lem}[\cite{Rousset-Tzvetkov JMPA 2008}]
Assume $\left|A(x,\sigma,j)-A_{\infty}(\sigma,j)\right|\leq Ce^{-\alpha|x|}$ and the spectrum of $A_{\infty}(\sigma,j)$ doesn't meet the imaginary axis for $\text{Re}(\sigma)>0$. Then 
\begin{align*}
|w|_s\leq C_{j,K,s}|H|_{s-1}.   
\end{align*}
\end{lem}
\begin{remark}
The statement of the lemma is slightly different from \cite[Lemma 4.2]{Rousset-Tzvetkov JMPA 2008}, but it is essentially the same.
\end{remark}
Based on the above statement, to prove Lemma \ref{tau small}, it suffices to show that the spectrum of $A_{\infty}(\sigma,j)$ doesn't intersect the imaginary axis for $\text{Re}(\sigma)>0$.
\begin{proof}[Proof of Lemma \ref{tau small}]
The characteristic polynomial of $A_{\infty}(\sigma,j)$  can be written as 
\begin{align*}
c\lambda^4-\sigma\lambda^3-(c-2\kappa)\lambda^2+\sigma\lambda+(jm_0k_0)^2,
\end{align*}
which doesn't have imaginary root for all $j$.
\end{proof}
Now we are ready to show \eqref{uj estimate} and thus Theorem \ref{u estmate}.
\begin{proof}[Proof of Theorem \ref{u estmate}]
By Theorem \ref{resolvent estimate} and Bessel-Parseval identity, for $T>0$
\begin{align*}
\int_0^Te^{-2\gamma_0t}|u_j(t)|_s^2dt & \leq \int_0^{\infty}e^{-2\gamma_0t}|\tilde{u}_j(t)|_s^2dt
=\int_{\mathbb{R}}|w(\tau)|_s^2d\tau\\
&\leq C\int_{\mathbb{R}}|H(\tau)|_{s-1}^2d\tau
=\int_0^Te^{-2\gamma_0t}\left|F_j(t)\right|_{s-1}^2dt. 
\end{align*}
From \eqref{fj estimate} we have
\begin{align}\label{integrate in time}
\int_0^Te^{-2\gamma_0t}|u_j(t)|_s^2dt\leq C\int_0^Te^{2(\gamma-\gamma_0)t}dt\leq Ce^{2(\gamma-\gamma_0)T}.
\end{align}
From \eqref{linear problem Fourier}, by the similar argument as in \eqref{rs 1}-\eqref{rs 3}, we can obtain the following $H^s$ estimate 
\begin{align*}
\begin{split}
\frac{d}{dt}|u_j(t)|_s^2
&\leq C\LC|u_j|_{s}^2+|F_j(t)|_{s-1}^2\RC \leq C|u_j(t)|_{s}^2+Ce^{2\gamma t},   
\end{split}
\end{align*}
and then
\begin{align*}
\frac{d}{dt}\LC e^{-2\gamma_0t}|u_j(t)|_s^2\RC\leq C\LC e^{-2\gamma_0t}|u_j(t)|_{s}^2+e^{2(\gamma-\gamma_0)t}\RC.  
\end{align*}
Integrating the above in time and by \eqref{integrate in time}, it follows that 
\begin{align*}
|u_j(t)|_{s}^2\leq Ce^{2\gamma t},     
\end{align*}
which proves \eqref{uj estimate}.
\end{proof}
\subsection{Error estimate and final result}\label{error estimate and final result}
In this subsection, we will first give an error estimate and then prove Theorem \ref{main theorem}. 

Let $u^{\delta}$ be decomposed as 
\begin{align*}
u^{\delta}=Q+v^{ap}+w.
\end{align*}
From \eqref{nonlinear equation}, $w$ satisfies 
\begin{equation}\label{error equation}
\left\{\begin{aligned}
\partial_tw=\ &\mathcal{J}\mathcal{L}w+\mathcal{J}
\LC\frac{1}{2}w_x^2+v^{ap}_xw_x+v^{ap}_{xx}w+(w+v^{ap})w_{xx}-\frac{3}{2}w^2-3v^{ap}w\RC\\
&+G,\\
w|_{t=0}\ &=0,
\end{aligned}
\right.
\end{equation}
where
\begin{align*}
G=-\partial_tv^{ap}+\mathcal{J}\mathcal{L}v^{ap}+\mathcal{J}\LC\frac{1}{2}(v^{ap}_x)^2+v^{ap}v^{ap}_{xx}-\frac{3}{2}(v^{ap})^2\RC.   
\end{align*}
The existence of $w$ in \eqref{error equation} will be proved in the Appendix \ref{appendix pf}. By Proposition \ref{v estimate}, we have
\begin{align}\label{approximate estimate}
\|G\|_s\leq C_{M,s}\delta^{M+2}e^{(M+2)\text{Re}(\sigma_0)t}.
\end{align}

The following {\it a priori} estimate will be crucial for the proof of the instability result.
\begin{lem}\label{error prior estimate}
Let $w\in H^s(\mathbb{R}\times\mathbb{T}_a)$ satisfy \eqref{error equation}, then 
\begin{align}\label{error estimate equation}
\frac{d}{dt}\|w\|_{s}^2\leq C_1\LC C_{2,M,s}+\|w\|_s\RC\|w\|_s^2+ C_{3,M,s}\delta^{2(M+2)}e^{2(M+2)\text{Re}(\sigma_0)t}.
\end{align}
\end{lem}
\begin{proof}
In this proof, we denote $\partial^k$ as the derivative with order $k$, $O(\partial^kw)$ as the polynomial differential operator on $w$ with highest order $k$, and $\LA \cdot,\cdot\RA $ as the inner product in $L^2(\mathbb{R}\times\mathbb{T}_a)$. It suffices to consider the estimate for the highest order $s$. Apply $\partial_x^{\alpha}\partial_y^{\beta}$ on \eqref{error equation} where $\alpha+\beta=s$ and take inner product with $\partial_x^{\alpha}\partial_y^{\beta}w$. Here we choose $s\geq 2$.

For the first term on the right-hand side of \eqref{error equation}, by \eqref{important transformation}
\begin{align}
&\LA \partial_x^{\alpha}\partial_y^{\beta}\mathcal{J}\mathcal{L}w,\partial_x^{\alpha}\partial_y^{\beta}w\RA \label{ee 1}\\
&\quad=\ \LA \partial_x^{\alpha}\partial_y^{\beta}\LC(c-\phi)w_x\RC+\mathcal{J}O(\partial_x^{\alpha}\partial_y^{\beta}w),\partial_x^{\alpha}\partial_y^{\beta}w\RA  \notag\\
&\quad=\ \LA  (c-\phi)\partial_x^{\alpha+1}\partial_y^{\beta}w+O(\partial_x^{\alpha}\partial_y^{\beta}w)+\mathcal{J}O(\partial_x^{\alpha}\partial_y^{\beta}w),\partial_x^{\alpha}\partial_y^{\beta}w\RA  \notag\\
&\quad=\ \LA  \frac{1}{2}\phi'\partial_x^{\alpha}\partial_y^{\beta}w+O(\partial_x^{\alpha}\partial_y^{\beta}w)+\mathcal{J}O(\partial_x^{\alpha}\partial_y^{\beta}w),\partial_x^{\alpha}\partial_y^{\beta}w\RA  \notag\\
&\quad\leq\ C\|w\|_s^2 \notag
\end{align}    
since $\mathcal{J}$ is bounded on $H^s(\mathbb{R}\times\mathbb{T}_a)$. 

For the second term on the right-hand side of \eqref{error equation}, rewrite it as 
\begin{align}
&\mathcal{J}\LC\frac{1}{2}w_x^2+v^{ap}_xw_x+v^{ap}_{xx}w+(w+v^{ap})w_{xx}-\frac{3}{2}w^2-3v^{ap}w\RC \label{ee 2}\\
&\quad=\ \LC 1-\partial_x^2\RC^{-1}\partial_x\LC-\frac{1}{2}w_x^2+v^{ap}_{xx}w+\LC(w+v^{ap})w_x\RC_x-\frac{3}{2}w^2-3v^{ap}w\RC \notag\\
&\quad=\ -\LC w+v^{ap}\RC w_x+\LC 1-\partial_x^2\RC^{-1}\partial_x\LC-\frac{1}{2}w_x^2+v^{ap}_{xx}w-\frac{3}{2}w^2-3v^{ap}w\RC, \notag
\end{align}
then
\begin{align}
&\LA \partial_x^{\alpha}\partial_y^{\beta}\LC (w+v^{ap})w_x\RC,\partial_x^{\alpha}\partial_y^{\beta}w\RA  \label{ee 3}\\
&\quad\leq\ \LA  (w+v^{ap})\partial_x^{\alpha+1}\partial_y^{\beta}w+w_x\partial_x^{\alpha}\partial_y^{\beta}(w+v^{ap}) \right. \notag\\
&\quad\quad +s\sum_{j+k=s-1}\partial^1(w+v^{ap})\partial_x^{j+1}\partial_y^kw 
+O(\partial^{[\frac{s}{2}]+1}(w+v^{ap}))O(\partial^{s-1}w),\partial_x^{\alpha}\partial_y^{\beta}w\Big>  \notag\\
&\quad\leq\ \LA  -\frac{1}{2}(w_x+v^{ap}_x)\partial_x^{\alpha}\partial_y^{\beta}w+w_x\partial_x^{\alpha}\partial_y^{\beta}(w+v^{ap}) \right.\notag\\
&\quad\quad \left. +s\sum_{j+k=s-1}\partial^1(w+v^{ap})\partial_x^{j+1}\partial_y^kw 
+O(\partial^{[\frac{s}{2}]+1}(w+v^{ap}))O(\partial^{s-1}w),\partial_x^{\alpha}\partial_y^{\beta}w \RA  \notag\\
&\quad\leq\ C\left\|O(\partial^{[\frac{s}{2}]+1}(w+v^{ap}))\right\|_{L^{\infty}}\|w\|_s^2 \ \leq\ C_1(C_{2,M,s}+\|w\|_s)\|w\|_s^2, \notag
\end{align}
and
\begin{align}
&-\frac{1}{2}\LA  \partial_x^{\alpha}\partial_y^{\beta}\LC\LC1-\partial_x^2\RC^{-1}\partial_x w_x^2\RC,\partial_x^{\alpha}\partial_y^{\beta}w\RA = \frac{1}{2}\LA  \partial_x^{\alpha}\partial_y^{\beta}w_x^2,\LC1-\partial_x^2\RC^{-1}\partial_x\partial_x^{\alpha}\partial_y^{\beta}w\RA  \label{ee 4} \\
&\quad\leq\ \Big<  w_x\partial_x^{\alpha+1}\partial_y^{\beta}w+2s\sum_{i+j=s-1}\partial^2w\partial_x^{j+1}\partial_y^{k}w \notag\\
&\quad\qquad \left. +O(\partial^{[\frac{s}{2}]+1}w) O(\partial^{s-1}w)),\LC 1-\partial_x^2\RC^{-1}\partial_x\partial_x^{\alpha}\partial_y^{\beta}w\RA  \notag\\
&\quad\leq\ \LA  w_x\partial_x^{\alpha+1}\partial_y^{\beta}w,\LC1-\partial_x^2\RC^{-1}\partial_x\partial_x^{\alpha}\partial_y^{\beta}w\RA +C\left\|O(\partial^{[\frac{s}{2}]+1}w)\right\|_{L^{\infty}}\|w\|_s^2 \notag\\
&\quad=\ -\LA  \partial_x^{\alpha}\partial_y^{\beta}w,w_x\LC1-\partial_x^2\RC^{-1}\partial_x^2\partial_x^{\alpha}\partial_y^{\beta}w+w_{xx}\LC1-\partial_x^2\RC^{-1}\partial_x\partial_x^{\alpha}\partial_y^{\beta}w\RA  \notag\\ 
&\quad\qquad+C\left\|O(\partial^{[\frac{s}{2}]+1}w)\right\|_{L^{\infty}}\|w\|_s^2 \notag\\
&\quad\leq\ -\LA  \partial_x^{\alpha}\partial_y^{\beta}w,w_x\LC1-\partial_x^2\RC^{-1}\LC\partial_x^2-1+1\RC\partial_x^{\alpha}\partial_y^{\beta}w\RA +C\left\|O(\partial^{[\frac{s}{2}]+1}w)\right\|_{L^{\infty}}\|w\|_s^2 \notag\\
&\quad\leq\ C\left\|O(\partial^{[\frac{s}{2}]+1}w)\right\|_{L^{\infty}}\|w\|_s^2 \ \leq\ C_1\LC C_{2,M,s}+\|w\|_s\RC\|w\|_s^2, \notag
\end{align}
and
\begin{align}
&\LA \partial_x^{\alpha}\partial_y^{\beta}\LC\LC1-\partial_x^2\RC^{-1}\partial_x\LC v^{ap}_{xx}w-\frac{3}{2}w^2-3v^{ap}w\RC\RC, \partial_x^{\alpha}\partial_y^{\beta}w\RA  \label{ee 5}\\
&\quad=\ \LA  \LC1-\partial_x^2\RC^{-1}\partial_x\partial_x^{\alpha}\partial_y^{\beta}\LC v^{ap}_{xx}w-\frac{3}{2}w^2-3v^{ap}w\RC,\partial_x^{\alpha}\partial_y^{\beta}w\RA  \notag\\
&\quad\leq\ C\left\|\partial_x^{\alpha}\partial_y^{\beta}\LC v^{ap}_{xx}w-\frac{3}{2}w^2-3v^{ap}w\RC\right\|_0\|w\|_s \notag\\
&\quad\leq\ C\left\|O(\partial^{[\frac{s}{2}]+1}w)\right\|_{L^{\infty}}\|w\|_s^2 \ \leq\ C_1\LC C_{2,M,s}+\|w\|_s\RC\|w\|_s^2. \notag
\end{align}
So by \eqref{approximate estimate}, \eqref{ee 1}-\eqref{ee 5}, the estimate \eqref{error estimate equation} is obtained.
\end{proof}
Now we give an error estimate. Let 
\begin{align*}
T^{\delta}=\frac{\log(\theta/\delta)}{\sigma_0},
\end{align*}
where $\theta$ will be chosen later.
Define $T^*$ such that 
\begin{align*}
T^*=\text{sup}\{T:T\leq T^{\delta} \text{ such that for any}\ t\in [0,T], \|w\|_s\leq 1\}.
\end{align*}
Then by Lemma \ref{error prior estimate}, when $0\leq t\leq T^*$,
\begin{align*}
\frac{d}{dt}\|w\|_{s}^2\leq C_1C_{2,M,s}\|w\|_s^2+ C_{3,M,s}\delta^{2(M+2)}e^{2(M+2)\text{Re}(\sigma_0)t}.  
\end{align*}
Note that $C_{2,M,s}$ is only related to $v^{ap}$. We can rewrite $C_{2,M,s}$ as $\theta C_{2,M,s}$ such that  the new $C_{2,M,s}$ is dependent on $s$ and $M$ but independent of $\theta$ and $t$. Then we have 
\begin{align*}
\frac{d}{dt}\|w\|_{s}^2\leq (C_1+\theta C_{2,M,s})\|w\|_s^2+ C_{3,M,s}\delta^{2(M+2)}e^{2(M+2)\text{Re}(\sigma_0)t}.     
\end{align*}
We can choose $M$ large enough and $\theta$ small enough such that 
\begin{align*}
    2(M+2)-C_1-\theta C_{2,M,s}>1.
\end{align*}
From now on, we fix $M$. Then by Gronwall's inequality we have 
\begin{align*}
\sup_{0\leq t\leq T^*}\|w\|_s\leq C_{M,s}\theta^{M+2}.
\end{align*}
When $\theta$ is sufficiently small, by the definition of $T^*$, we actually have $T^*=T^{\delta}$, i.e.
\begin{align*}
\sup_{0\leq t\leq T^{\delta}}\|w\|_s\leq C_{M,s}\theta^{M+2}
\end{align*}
for $s\geq 2$. 
In particular, we have 
\begin{align}\label{error estimate}
\left\|w(T^{\delta},\cdot)\right\|_0\leq C_{M,s}\theta^{M+2}.    
\end{align}

Now we are in the position to prove Theorem \ref{main theorem}.
\begin{proof}[Proof of Theorem \ref{main theorem}]
Denote $\Pi$ the projection onto the zero mode in $y$, i.e.
\begin{align*}
\Pi \LC u(x,y)\RC = u(x,y)-\int_0^{\frac{2\pi}{k_0}}u(x,y)dy,    
\end{align*}
then 
\begin{align*}
\begin{split}
    \left\|\Pi (v^{ap})\right\|_0 &\geq \delta\left\|\Pi (v_0)\right\|_0-\sum_{k=1}^{M}\delta^{k+1}\left\|\Pi (v^k)\right\|_0 \\
    &=\delta\|v_0\|_0-\sum_{k=1}^{M}\delta^{k+1}\left\|\Pi (v^k)\right\|_0\\
    &\geq c_s\delta e^{\sigma_0t}-\sum_{k=1}^MC_{k,s}\delta^{k+1}e^{(k+1)\sigma_0t}.
\end{split}
\end{align*}
When $\theta$ is sufficiently small, we have 
\begin{align}\label{projection estimate}
    \left\|\Pi (v^{ap}(T^{\delta},\cdot))\right\|_0\geq \frac{c_s\theta}{2}.
\end{align}
Then by \eqref{error estimate} and \eqref{projection estimate}, for any $l\in \mathbb{R}$,
\begin{align*}
    \begin{split}
        \left\|u^{\delta}(T^{\delta},\cdot)-\phi(\cdot-l)\right\|_0
        &\geq \left\|\Pi \left(u^{\delta}(T^{\delta},\cdot)-\phi(\cdot-l)\right)\right\|_0\\
        &=\left\|\Pi \left(u^{\delta}(T^{\delta},\cdot)-\phi(\cdot)\right)\right\|_0\\
        &=\left\|\Pi\left(v^{ap}(T^{\delta},\cdot)+w(T^{\delta},\cdot)\right)\right\|_0\\
        &\geq \frac{c_s\theta}{2}-\left\|\Pi\left(w(T^{\delta},\cdot)\right)\right\|_0\\
        &\geq \frac{c_s\theta}{2}-\left\|w(T^{\delta},\cdot)\right\|_0\\
        &\geq \frac{c_s\theta}{2}-C_{M,s}\theta^{M+2},
    \end{split}
\end{align*}
when $\theta$ is chosen appropriately, the estimate will be bounded below by a fixed $\eta$ depending only on $s$, which proves the theorem .
\end{proof}

\section*{Acknowledgements}
R. M. Chen and J. Jin are supported in part by the NSF grants DMS-1907584.

\appendix
\section{Proofs}
\subsection{Proof of Lemma \ref{well posedness u_j}}\label{appendix lem}
\begin{proof}[Proof of Lemma \ref{well posedness u_j}]
By Duhamel's principle, it suffices to prove the existence of solution for the homogeneous equation:
\begin{align*}
\partial_tu=\mathcal{J}\mathcal{L}(jm_0k_0)u,\ \  u|_{t=0}=\tilde{u}\ \ \ \text{for}\ j=1\cdots, K.   
\end{align*}
Since
\begin{align}
\mathcal{J}\mathcal{L}(jm_0k_0)=\ &\LC1-\partial_x^2\RC^{-1}\LC-\partial_x^2\RC\LC(c-\phi)u_x\RC \label{JL decomposition}\\
&+\LC1-\partial_x^2\RC^{-1}\partial_x\LC\LC\phi''-3\phi-2\kappa+c\RC u\RC
-(jm_0k_0)^2\LC1-\partial_x^2\RC^{-1}\partial_x^{-1}u \notag\\
=\ &(c-\phi)u_x-\LC1-\partial_x^2\RC^{-1}\LC(c-\phi)u_x\RC \notag\\
&+\LC1-\partial_x^2\RC^{-1}\partial_x\LC\LC\phi''-3\phi-2\kappa+c\RC u\RC
-(jm_0k_0)^2\LC1-\partial_x^2\RC^{-1}\partial_x^{-1}u,  \notag 
\end{align}
it suffices to study the operator
\begin{align*}
\mathcal{A}=(c-\phi)\partial_x-(jm_0k_0)^2\LC1-\partial_x^2\RC^{-1}\partial_x^{-1}   
\end{align*}
since other terms are just bounded perturbations.

Consider $\mathcal{A}: H^{s+1}(\mathbb{R}) \cap   \mathcal{D}\LC\partial_x^{-1}(\mathbb{R})\RC\rightarrow H^s(\mathbb{R})$, where $\mathcal{D}(\partial_x^{-1}(\mathbb{R})) = \mathcal{F}^{-1}\left\{u:\hat{u}(0)=0\right\}$ and $\mathcal{F}$ is the Fourier transform with respect to $x$. We first prove that 
\begin{align}\label{dissipation estimate}
\LA  \mathcal{A}u,u\RA _{H^s} \leq \omega\LA  u,u\RA _{H^s}    
\end{align}
for some $\omega>0$. For $(c-\phi)\partial_x$,
\begin{align*}
\begin{split}
\LA  (c-\phi)u_x,u\RA _{H^s}&=\LA  \partial_x^s((c-\phi)u_x),\partial_x^s u\RA +\sum_{\alpha=0}^{s-1} \LA  \partial_x^{\alpha}((c-\phi)u_x),\partial_x^{\alpha} u\RA \\
&\leq \LA  \partial_x^s((c-\phi)u_x),\partial_x^s u\RA  + \omega_1\LA  u,u\RA _{H^s}.
\end{split}
\end{align*}
It reduces to control order $s$ term, and we have
\begin{align*}
\begin{split}
\LA  \partial_x^s((c-\phi)u_x),\partial_x^s u\RA  &= 
\LA  (c-\phi)\partial_x^{s+1}u+O(\partial_x^su),\partial_x^su\RA \\
&=\LA  \frac{1}{2}\phi'\partial_x^su+O(\partial_x^su),\partial_x^su\RA \\
&\ \leq\omega_2\LA  u,u\RA _{H^s}.    
\end{split}
\end{align*}
For $\LC1-\partial_x^2\RC^{-1}\partial_x^{-1}$,
\begin{align*}
\begin{split}
\LA \LC1-\partial_x^2\RC^{-1}\partial_x^{-1}u,u\RA _{H^s}&= \LA  \LC1-\partial_x^2\RC^{-1}\partial_x^{-1}u,u\RA +\sum_{\alpha=1}^s\LA  \partial_x^{\alpha}\LC\LC1-\partial_x^2\RC^{-1}\partial_x^{-1}u\RC,\partial_x^{\alpha}u\RA \\
&\leq \LA  \LC1-\partial_x^2\RC^{-1}\partial_x^{-1}u,u\RA +\omega_3\LA  u,u\RA _{H^s}\\
&= 0+\omega_3\LA  u,u\RA _{H^s}.
\end{split}
\end{align*}

Next we prove that $\lambda-(\mathcal{A}-\omega)$ is surjective for $\lambda>0$. Since by \eqref{dissipation estimate}, there is no point spectrum larger than $0$. It suffices to prove that $\lambda > 0$ is not in the essential spectrum of $\mathcal{A}-\omega$. It is enough just to consider the essential spectrum of its limiting operator 
\begin{align*}
c\partial_x-(jm_0k_0)^2(1-\partial_x^2)^{-1}\partial_x^{-1}-\omega.
\end{align*}
By using Fourier transform it is clear that $\lambda > 0$ is not in the essential spectrum of the above operator.
Based on all the above, by Lumer-Phillips theorem \cite{Engel Nagel 2006}, the lemma is concluded.
\end{proof}
\subsection{Existence of solution in \eqref{error equation}}\label{appendix pf}
\begin{proof}
The proof follows the idea of \cite{Constantin Escher 1998, Molinet 2004}.
Consider the regularized problem
\begin{equation}\label{regularized equation}
\left\{\begin{aligned}
\partial_tw^{\varepsilon}=\ &\mathcal{J}^{\varepsilon}
\LC\frac{1}{2}(w_x^{\varepsilon})^2+v^{ap}_xw_x^{\varepsilon}+v^{ap}_{xx}w^{\varepsilon}+(w^{\varepsilon}+v^{ap})w_{xx}^{\varepsilon}-\frac{3}{2}(w^{\varepsilon})^2-3v^{ap}w^{\varepsilon}\RC\\
&+\mathcal{J}^{\varepsilon}\mathcal{L}w^{\varepsilon}+G^{\varepsilon}, \\
w|_{t=0}\ &=0,
\end{aligned}
\right.
\end{equation}
where 
\begin{align*}
\mathcal{J}^{\varepsilon}=\LC 1-\partial_x^2+\varepsilon\Delta^2\RC^{-1}\partial_x,\ \Delta=\partial_x^2+\partial_y^2.  
\end{align*}
It can be derived from fixed point argument that the solution $w_{\varepsilon}$ exists. Indeed, since $\mathcal{J}^{\varepsilon}$ maps $H^{s}\rightarrow H^{s+3}$, it suffices to choose a unit ball in $C([0,t^{\varepsilon}])$ for the contraction mapping. Next we state the approximating procedure.
We could have the same estimate as \eqref{error estimate equation} for $w^{\varepsilon}$:
\begin{align*}
\frac{d}{dt}\|w^{\varepsilon}\|_{s}^2\leq C_1\LC C_{2,M,s}+\|w^{\varepsilon}\|_s\RC\|w^{\varepsilon}\|_s^2+ C_{3,M,s}\delta^{2(M+2)}e^{2(M+2)\text{Re}(\sigma_0)t}.    
\end{align*}
Then for each $\varepsilon$, we define $T^{\varepsilon}$
\begin{align*}
T^{\varepsilon}=\sup\{T:\|w\|_s\leq C_{2,M,s}\ \text{for}\ 0\leq t\leq T\}.
\end{align*}
Then for each $\varepsilon$ such that  $T^{\varepsilon}<1$, we have\\
\begin{align*}
\frac{d}{dt}\left\|w^{\varepsilon}\right\|_{s}^2\leq 2C_1\left\|w^{\varepsilon}\right\|_s\left\|w^{\varepsilon}\right\|_s^2+ C_{3,M,s}\delta^{2(M+2)}e^{2(M+2)\text{Re}(\sigma_0)t} ,  
\end{align*}
which yields
\begin{align*}
\begin{split}
\|w^{\varepsilon}\|_s^2\leq &\bigg(\frac{1}{\sqrt{C_{2,M,s}}}-2C_1(t-T^{\varepsilon})\bigg)^{-2}\\
&+\int_{T^{\varepsilon}}^t\bigg(\frac{1}{\sqrt{C_{3,M,s}\delta^{2(M+2)}e^{2(M+2)\text{Re}(\sigma_0)s}}}-2C_1(t-s)\bigg)^{-2}ds.
\end{split}
\end{align*}
Since $T^{\varepsilon}<1$, for $t$ sufficiently close to $T^{\varepsilon}$, 
\begin{align*}
\frac{1}{\sqrt{C_{3,M,s}\delta^{2(M+2)}e^{2(M+2)\text{Re}(\sigma_0)s}}}-2C_1(t-s)>c>0
\end{align*}
for all $\varepsilon$ such that  $T^{\varepsilon}<1$. So there exists $T$ such that  $\|w^{\varepsilon}\|_s^2$ is uniformly bounded on $[0,T]$ for all $\varepsilon$ when $T^{\varepsilon}<1$. Then for all $\varepsilon$, $\|w^{\varepsilon}\|_s^2$ is uniformly bounded on $[0,\tilde{T}]$ where $\tilde{T}=\min (T,1)$. And from \eqref{regularized equation}, $\{\partial_tw^{\varepsilon}\}$ is uniformly bounded in $L^{\infty}([0,\tilde{T}];L^2_{\mathbb{R}\times \mathbb{T}_a})$. Then by Aubin-Lions lemma, we obtain a solution $u\in L^{\infty}([0,\tilde{T}],H^s_{\mathbb{R}\times \mathbb{T}_a})$ for \eqref{error equation}.
\end{proof}

\end{document}